\documentclass[11pt,twoside]{amsart}
\usepackage{amsmath, amsthm, amscd, amsfonts, amssymb, graphicx, color}
\usepackage[bookmarksnumbered, colorlinks, plainpages]{hyperref}
\usepackage{pgf,tikz}
\usetikzlibrary{arrows}
\usepackage[utf8]{inputenc}
\usepackage{pstricks-add}
\input{mathrsfs.sty}

\newtheorem{theorem}{Theorem}[section]

\newtheorem{proposition}[theorem]{Proposition}
\newtheorem{corollary}[theorem]{Corollary}
\theoremstyle{definition}

\newtheorem{example}[theorem]{Example}

\numberwithin{equation}{section}

\begin{document}

\title{Cohen-Macaulay $r$-partite graphs with minimal clique cover}

\author[A. Madadi, R. Zaare-Nahandi]{Asghar Madadi, Rashid Zaare-Nahandi}
\vspace{-2cm}
\address{Asghar Madadi,  Department of Mathematics, Faculty of Sciences, University of Zanjan, P. O. Box 45195-313, Zanjan, Iran}

\email{a$\_$madadi@znu.ac.ir}

\address{Rashid Zaare-Nahandi,  Department of  Mathematics, Institute for Advanced Studies in Basic Sciences, P. O. Box 45195-1159, Zanjan, Iran}

\email{rashidzn@iasbs.ac.ir}



\begin{abstract}
In this note, we give some necessary conditions for an $r$-partite graph such that the edge ring of the graph is Cohen-Macaulay. It is proved that if $G$ is an $r$-partite Cohen-Macaulay graph which is covered by some disjoint cliques of size $r$, then the clique cover is unique.
\end{abstract} \maketitle

\section{Introduction}
Mainly, after using the notion of simplicial complexes and its algebraic interpretation by R. Stanley in 1970s to prove the upper bound conjecture for number of simplicial spheres \cite{S}, this notion has been one of the main streams of research in commutative algebra. In this stream, characterization and classification of Cohen-Macaulay simplicial complexes have been extensively studied in the last decades. It is known that Cohen-Macaulay property of a simplicial complex and its level graph are coincide. Therefore, to characterize  all simplicial complexes which are Cohen-Macaulay, is enough to characterize all graphs with this property \cite{S}.

To examine special classes of graphs, Estrada and Villarreal in \cite{EV} found some necessary conditions for bipartite graphs to be Cohen-Macaulay. Finally, Herzog and Hibi in \cite{HH} presented a combinatorial characterization for bipartite graphs equivalent to Cohen-Macaulay property of these graphs. This purely combinatorial method can not be generalized for $r$-partite graphs in general. Because, as shown in Example~\ref{Ex}, Cohen-Macaulay property may depend on characteristics of the base field. In other hand, it is shown in  \cite{Z}, the corresponding graph to a simplicial complex, such that has the same Cohen-Macaulayness property is covered by minimal possible number of cliques. In this paper, we consider $r$-partite graphs with a minimal clique cover and find a necessary condition for Cohen-Macaualayness of these graphs. More precisely, we prove that in a Cohen-Macaulay $r$-partite graph with a minimal clique cover, there is a vertex of degree $r-1$ and the cover is unique.

\section{Preliminaries}

A simple graph is an undirected graph that has no loop and multiple edge. A finite graph is denoted by $G = (V(G), E(G))$, where $V(G)$ is the set of vertices and $E(G)$ is the set of edges. Let $|V(G)|=n$. We use $[n]=\{1,2,\ldots,n\}$ as vertices of $G$. The complementary graph of  $G$  is the graph $\bar{G}$ on $[n]$ whose edge set $E(\bar{G})$ consists of those edges $\{i,j\}$ which are not in $ E(G)$.  An independent set of vertices is a set of pairwise nonadjacent vertices. An $r$-partite graph is a graph that the set of its vertices can be partitioned into $r$ disjoint subsets such that each set is independent. A subset $A \subset [n]$ is a minimal vertex cover of $G$ if (i) each edge of $G$ is incident with at least one vertex in $A$, and (ii) there is no proper subset of $A$
with  property (i). It is easy to check that any minimal vertex cover of a graph is complement set of a maximal independent set of the graph.  A graph $G$ is called unmixed (well-covered) if any two minimal vertex covers of $G$ have the same cardinality. A clique in a graph is a set of pairwise adjacent vertices, and by a $r$-clique we mean a clique of size $r$. An $r$-matching is a set of pairwise disjoint $r$-cliques and a perfect $r$-matching is an $r$-matching which covers all vertices of $G$.

Let $\omega(G)$ denote the maximum size of cliques in $G$, which is called clique number of $G$. Let $f: V(G) \to [k]$ be a
map such that if $v_1$ is adjacent to $v_2$ then $f(v_1) \neq f(v_2)$. If such a map exists, we say that $G$ is colorable
by $k$ colors. The smallest such $k$ is called chromatic number of the graph and is denoted by  $\chi(G)$. A graph $G$ is called perfect if $\omega(H)=\chi(H)$ for each induced subgraph $H$ of $G$. The class of perfect graphs plays an important role in graph theory and most of computations
in this class can be done by fast algorithms. L. Lov\'{a}sz in \cite{L} has proved that a graph is perfect if and only if its
complement is perfect. Chudnovsky et al in \cite{4} have  proved that a necessary and sufficient condition for a graph $G$ to be perfect is that $G$
does not have an odd hole (a cycle of odd length greater than 3) or an odd antihole (complement of an odd hole) as induced subgraph.

Let $G$ be a graph on $[n]$. Let $S=K[x_1,\ldots,x_n]$, the polynomial ring over a field $K$. The edge ideal $I(G)$ of $G$ is defined to be the ideal of $S$ generated by all square-free monomials $x_ix_j$ provided that $i$ is adjacent to $j$  in $G$. The quotient ring $R(G)=S/I(G)$ is called the edge ring of $G$.

Let $R$ be a commutative ring with an identity. The depth of $R$, denoted by $depth(R)$, is the largest integer $r$ such that there is a sequence $f_1,\ldots,f_r$ of elements of $R$ such that $f_i$ is not a zero-divisor in $R/(f_1,\ldots,f_{i-1})$ for all $1 \leq i \leq r$, and
$(f_1,\ldots,f_r) \neq R$. Such a sequence is a called a regular sequence. The depth is an important invariant of a ring. It is bounded by another important invariant, the Krull dimension, the length of the longest chain of prime ideals in the ring. A ring $R$ is called Cohen-Macaulay if $depth(R)=dim(R)$. A graph $G$ is called Cohen-Macaulay if the ring $R(G)$ is Cohen-Macaulay.
%--------------------------------------------------------------------------------------------------------------
\begin{theorem}  \cite[Proposition 6.1.21]{V}\label{V1}
If $G$ is a Cohen-Macaulay graph, then $G$ is unmixed.
\end{theorem}
%---------------------------------------------------------------------------------------------------------------
A  simplicial complex $\Delta$ on $n$ vertices is a collection of subsets of $[n]$ such that the following conditions hold:
\\(i) $\{i\} \in \Delta$ for each $i \in [n]$,
\\(ii) if $E \in \Delta$ and $F \subseteq E$ then $F \in \Delta$.
\\An element of $\Delta$ is called a face and a maximal face with respect to inclusion is called a facet. The set of all facets of $\Delta$ is denoted by $\mathcal{F}(\Delta)$. The dimension of a face $F \in \Delta$ is defined to be $|F|-1$ and dimension of $\Delta$ is maximum of dimension of its faces. A simplicial complex is called pure if all of its facets have the same dimension. For more details on simplicial complexes see \cite{S}.

The clique complex of a finite graph $G$ on $[n]$ is the simplicial complex $\Delta(G)$ on $[n]$ whose faces are the cliques of $G$.
Let $\Delta$ be a simplicial complex on $[n]$. We say that $\Delta$  is shellable if its facets can be ordered as $F_1,F_2,\ldots,F_m$ such that for all $j \geq 2$ the subcomplex $(F_1,\ldots,F_{j-1}) \cap F_j$ is pure of dimension dim$F_j-1$. An order of the facets satisfying this condition is called a shelling order. To say that $F_1,F_2,\ldots,F_m$ is a shelling order of $\Delta$ is equivalent to say that for all $i$, $2 \leq i \leq m$ and all $j < i$, there exists $l \in F_i \setminus F_j$ and $k<i$ such that $F_i \setminus F_k=\{l\}$. $G$ is called shellable if $\Delta(\bar{G})$ has this property.

Let $\Delta$ be a simplicial complex on $[n]$ and $I_{\Delta}$ be the ideal of $S=K[x_1,\ldots,x_n]$ generated by all square-free monomials $x_{i_1} \cdots x_{i_t}$, provided that $\{i_1,\ldots,i_t\}$ is not a face of $\Delta$. The ring $S/I_{\Delta}$ is called the Stanley-Reisner ring of $\Delta$. A simplicial complex is called Cohen-Macaulay if its Stanley-Reisner ring is Cohen-Macaulay.
%-----------------------------------------------------------------------------------------------------------------
\begin{theorem} \cite[Theorem 8.2.6]{HH2}
If $\Delta$ is a pure and shellable simplicial complex, then $\Delta$ is Cohen-Macaulay.
\end{theorem}
%----------------------------------------------------------------------------------------------------------------------
M. Estrada and R. H. Villarreal in \cite{EV} have proved that for a bipartite graph $G$ Cohen-Macaulayness and pure shellability are equivalent. This is not true  in general for $r$-partite graphs when $r>2$ (Example~\ref{Ex}).

Also in bipartite graphs, Cohen-Macaulayness does not depend on characteristics of the ground field. But again, this is not true in general as shown in the following example.
%---------------------------------------------------------------------------------------------------------------------------------------
\begin{example}\label{Ex}
Let $G$ be the graph in Figure~\ref{ch}. Then, $R(G)$ is Cohen-Macaulay when the  characteristic of the ground field $K$ is zero but it is not Cohen-Macaulay in characteristic 2. Therefore the graph $G$ is not shellable (\cite{K}).

\begin{figure}
\begin{tikzpicture}[line cap=round,line join=round,>=triangle 45,x=0.8cm,y=0.8cm]
\clip(-6.6,-2.5) rectangle (14,3);
\draw (-2,2)-- (0,2.5);
\draw (-2,2)-- (0,0.5);
\draw (-2,2)-- (2,0);
\draw (-2,2)-- (2,-2);
\draw (-2,0)-- (0,0.5);
\draw (-2,0)-- (0,-1.5);
\draw (-2,0)-- (2,0);
\draw (-2,0)-- (4,2.5);
\draw (-2,0)-- (4,0.5);
\draw (-2,-2)-- (0,-1.5);
\draw (-2,-2)-- (2,2);
\draw (-2,-2)-- (2,-2);
\draw (-2,-2)-- (4,2.5);
\draw (0,2.5)-- (2,2);
\draw (0,2.5)-- (2,0);
\draw (0,2.5)-- (4,0.5);
\draw (0,0.5)-- (2,-2);
\draw (0,0.5)-- (4,2.5);
\draw (0,0.5)-- (4,0.5);
\draw (0,-1.5)-- (2,0);
\draw (0,-1.5)-- (2,-2);
\draw (0,-1.5)-- (4,0.5);
\draw (2,2)-- (4,2.5);
\draw (2,2)-- (4,0.5);
\draw (2,-2)-- (4,0.5);
\fill [color=black] (-2,0) circle (1.5pt);
\draw[color=black] (-2.29,0) node {2};
\fill [color=black] (-2,2) circle (1.5pt);
\draw[color=black] (-2.27,2) node {1};
\fill [color=black] (-2,-2) circle (1.5pt);
\draw[color=black] (-2.25,-2) node {3};
\fill [color=black] (0,0.5) circle (1.5pt);
\draw[color=black] (-0.01,0.77) node {5};
\fill [color=black] (0,2.5) circle (1.5pt);
\draw[color=black] (-0.03,2.81) node {4};
\fill [color=black] (0,-1.5) circle (1.5pt);
\draw[color=black] (-0.03,-1.16) node {6};
\fill [color=black] (2,0) circle (1.5pt);
\draw[color=black] (2.21,0) node {8};
\fill [color=black] (2,2) circle (1.5pt);
\draw[color=black] (2.04,2.29) node {7};
\fill [color=black] (2,-2) circle (1.5pt);
\draw[color=black] (2.23,-2) node {9};
\fill [color=black] (4,2.5) circle (1.5pt);
\draw[color=black] (4.26,2.5) node {10};
\fill [color=black] (4,0.5) circle (1.5pt);
\draw[color=black] (4.3,0.5) node {11};
\end{tikzpicture}
%\end{document}
\caption{Cohen-Macaulay property depends on characteristic}\label{ch}
\end{figure}
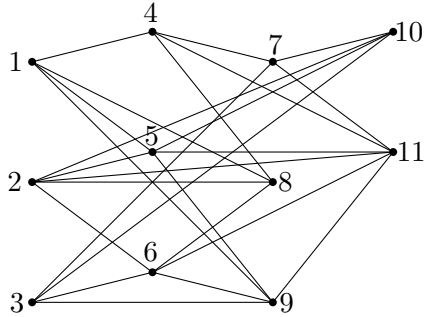

\end{example}

\section{Cohen-Macaulay property and uniqueness of perfect $r$-matching}

M. Estrada and R. H. Villarreal in \cite{EV} have proved that if $G$ is a Cohen-Macaulay bipartite graph and has at least one vertex of positive degree, then there is a vertex $v$ such that $\deg(v)=1$. By $\deg(v)$ we mean the number of vertices adjacent to $v$.
J. Herzog and T. Hibi in \cite{HH} have proved that a bipartite graph $G$ with parts $V_1$ and $V_2$ is Cohen-Macaulay if and only if, $|V_1|=|V_2|$ and there is an order on the vertices of $V$ and $W$ as $v_1,\ldots,v_n$ and $w_1,\ldots,w_n$ respectively, such that:
\\1) $v_i\sim w_i$ for $i=1,\ldots,n$,
\\2) if $v_i\sim w_j$, then $i \leq j$,
\\3) for each $1 \leq i<j<k \leq n$ if $v_i \sim w_j$ and $v_j \sim w_k$, then $v_i \sim w_k$.

R. Zaare-Nahandi in \cite{Z1} has proved that a well-covered bipartite graph $G$ is Cohen-Macaulay if and only if there is a unique perfect $2$-matching in $G$.

Let $\alpha(G)$ denote the maximum cardinality of independent sets of vertices of $G$.
Let $\mathcal{G}$ be the class of graphs such that for each $G \in \mathcal{G}$ there are $k=\alpha(G)$ cliques in $G$ covering all its vertices. For each
$G \in \mathcal{G}$ and cliques $Q_1,\ldots,Q_k$ such that $V(Q_1)\cup \cdots \cup V(Q_k)=V(G)$, we may take ${Q'}_1=Q_1$ and for $i=2,\ldots,k$, ${Q'}_i$ the
induced subgraph on the vertices $V(Q_i)\setminus (V(Q_1)\cup \cdots \cup V(Q_{i-1}))$. Then ${Q'}_1,\ldots,{Q'}_k$  are $k$ disjoint cliques covering all  vertices
of $G$. We call such a set of cliques, a basic clique cover of the graph $G$. Therefore any graph in the class $\mathcal{G}$ has a basic clique cover.

\begin{proposition}
Let $G$ be an $r$-partite, unmixed and perfect graph such that all maximal cliques are of size $r$. Then $G$ is in the class $\mathcal{G}$.
\end{proposition}
\begin{proof}
Let $V_1,\ldots,V_r$ be parts of $G$.  By \cite{Z}, $|V_1|=|V_2|=\cdots=|V_r|=\alpha(G)$. Also by \cite{L}, the complement graph $\bar{G}$ is perfect. In other hand, $V_i$ is a clique of maximal size in $\bar{G}$ for each $1 \leq i \leq r$. Therefore,
$\chi(\bar{G})=\omega(\bar{G})=\alpha(G)$. This implies that $\bar{G}$ is $\alpha(G)$-partite. Therefore there are $\alpha(G)$ disjoint maximal cliques in $G$ covering all vertices.
\end{proof}
The converse of the above proposition is not true in the sense of the following example.

\begin{example}
Let $G$ be the graph in Figure~\ref{per}. Then $G$ is a graph in class $\mathcal{G}$ which is $4$-partite, unmixed and all maximal cliques are of size $4$. But the induced subgraph on  $\{A,B,C,D,E\}$ is a cycle of length $5$ and therefore, by \cite{4}, the graph $G$ is not perfect.

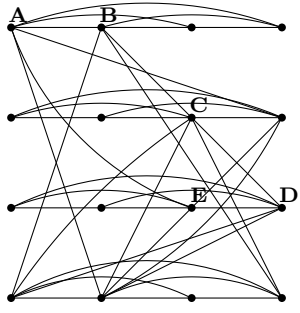
\begin{figure}

\definecolor{qqqqff}{rgb}{0,0,1}
\begin{tikzpicture}[line cap=round,line join=round,>=triangle 45,x=0.6cm,y=0.6cm]
%\psset{xunit=0.7cm,yunit=0.7cm,algebraic=true,dotstyle=o,dotsize=3pt 0,linewidth=0.8pt,arrowsize=3pt 2,arrowinset=0.25}
\clip(.5,-0.5) rectangle (10.98,7.28);
\draw (3,0)-- (5,0);
\draw (5,0)-- (7,0);
\draw (7,0)-- (9,0);
\draw (3,2)-- (5,2);
\draw (5,2)-- (7,2);
\draw (7,2)-- (9,2);
\draw (3,4)-- (5,4);
\draw (5,4)-- (7,4);
\draw (7,4)-- (9,4);
\draw (3,6)-- (5,6);
\draw (5,6)-- (7,6);
\draw (7,6)-- (9,6);
\draw [shift={(6,-3)}] plot[domain=1.17:1.98,variable=\t]({1*7.62*cos(\t r)+0*7.62*sin(\t r)},{0*7.62*cos(\t r)+1*7.62*sin(\t r)});
\draw [shift={(6,-4)}] plot[domain=1.11:2.03,variable=\t]({1*6.71*cos(\t r)+0*6.71*sin(\t r)},{0*6.71*cos(\t r)+1*6.71*sin(\t r)});
\draw [shift={(6,-5)}] plot[domain=1.03:2.11,variable=\t]({1*5.83*cos(\t r)+0*5.83*sin(\t r)},{0*5.83*cos(\t r)+1*5.83*sin(\t r)});
\draw [shift={(6,-2)}] plot[domain=1.21:1.93,variable=\t]({1*8.54*cos(\t r)+0*8.54*sin(\t r)},{0*8.54*cos(\t r)+1*8.54*sin(\t r)});
\draw [shift={(5,-1)}] plot[domain=1.29:1.85,variable=\t]({1*7.28*cos(\t r)+0*7.28*sin(\t r)},{0*7.28*cos(\t r)+1*7.28*sin(\t r)});
\draw [shift={(5,-2)}] plot[domain=1.25:1.89,variable=\t]({1*6.32*cos(\t r)+0*6.32*sin(\t r)},{0*6.32*cos(\t r)+1*6.32*sin(\t r)});
\draw [shift={(5,-3)}] plot[domain=1.19:1.95,variable=\t]({1*5.39*cos(\t r)+0*5.39*sin(\t r)},{0*5.39*cos(\t r)+1*5.39*sin(\t r)});
\draw [shift={(5,-4)}] plot[domain=1.11:2.03,variable=\t]({1*4.47*cos(\t r)+0*4.47*sin(\t r)},{0*4.47*cos(\t r)+1*4.47*sin(\t r)});
\draw [shift={(7,-1)}] plot[domain=1.29:1.85,variable=\t]({1*7.28*cos(\t r)+0*7.28*sin(\t r)},{0*7.28*cos(\t r)+1*7.28*sin(\t r)});
\draw [shift={(7,-2)}] plot[domain=1.25:1.89,variable=\t]({1*6.32*cos(\t r)+0*6.32*sin(\t r)},{0*6.32*cos(\t r)+1*6.32*sin(\t r)});
\draw [shift={(7,-3)}] plot[domain=1.19:1.95,variable=\t]({1*5.39*cos(\t r)+0*5.39*sin(\t r)},{0*5.39*cos(\t r)+1*5.39*sin(\t r)});
\draw [shift={(7,-4)}] plot[domain=1.11:2.03,variable=\t]({1*4.47*cos(\t r)+0*4.47*sin(\t r)},{0*4.47*cos(\t r)+1*4.47*sin(\t r)});
\draw (3,6)-- (5,0);
\draw (3,6)-- (9,4);
\draw (5,6)-- (7,4);
\draw (7,4)-- (9,2);
\draw (5,6)-- (3,0);
\draw (5,6)-- (9,0);
\draw (7,4)-- (9,0);
\draw (9,4)-- (7,2);
\draw (7,2)-- (5,0);
\draw (5,0)-- (9,2);
\draw (9,2)-- (3,0);
\draw (7,4)-- (5,0);
\draw [shift={(14.9,-7.9)}] plot[domain=2.16:2.56,variable=\t]({1*14.28*cos(\t r)+0*14.28*sin(\t r)},{0*14.28*cos(\t r)+1*14.28*sin(\t r)});
\draw [shift={(1,8)}] plot[domain=5.18:5.82,variable=\t]({1*8.94*cos(\t r)+0*8.94*sin(\t r)},{0*8.94*cos(\t r)+1*8.94*sin(\t r)});
\draw [shift={(9.19,8.19)}] plot[domain=3.48:4.37,variable=\t]({1*6.57*cos(\t r)+0*6.57*sin(\t r)},{0*6.57*cos(\t r)+1*6.57*sin(\t r)});
\begin{scriptsize}
\fill [color=black] (3,0) circle (1.5pt);
%\draw[color=qqqqff] (3.14,0.28) node {$A$};
\fill [color=black] (5,0) circle (1.5pt);
%\draw[color=qqqqff] (5.16,0.28) node {$B$};
\fill [color=black] (7,0) circle (1.5pt);
%\draw[color=qqqqff] (7.16,0.28) node {$C$};
\fill [color=black] (9,0) circle (1.5pt);
%\draw[color=qqqqff] (9.16,0.28) node {$D$};
\fill [color=black] (3,2) circle (1.5pt);
%\draw[color=qqqqff] (3.16,2.28) node {$E$};
\fill [color=black] (5,2) circle (1.5pt);
%\draw[color=qqqqff] (5.14,2.28) node {$F$};
\fill [color=black] (7,2) circle (1.5pt);
\draw[color=black] (7.16,2.28) node {\bf E};
\fill [color=black] (9,2) circle (1.5pt);
\draw[color=black] (9.16,2.28) node {\bf D};
\fill [color=black] (3,4) circle (1.5pt);
%\draw[color=qqqqff] (3.1,4.28) node {$I$};
\fill [color=black] (5,4) circle (1.5pt);
%\draw[color=qqqqff] (5.14,4.28) node {$J$};
\fill [color=black] (7,4) circle (1.5pt);
\draw[color=black] (7.16,4.28) node {\bf C};
\fill [color=black] (9,4) circle (1.5pt);
%\draw[color=qqqqff] (9.14,4.28) node {$L$};
\fill [color=black] (3,6) circle (1.5pt);
\draw[color=black] (3.16,6.28) node {\bf A};
\fill [color=black] (5,6) circle (1.5pt);
\draw[color=black] (5.16,6.28) node {\bf B};
\fill [color=black] (7,6) circle (1.5pt);
%\draw[color=qqqqff] (7.16,6.28) node {$O$};
\fill [color=black] (9,6) circle (1.5pt);
%\draw[color=qqqqff] (9.16,6.28) node {$P$};
\end{scriptsize}
\end{tikzpicture}
\caption{A graph in class $\mathcal G$ which is not perfect}\label{per}
\end{figure}

\end{example}
%--------------------------------------------------------------------------------------------------------------------

Let $H$ be a graph and $v$ be a vertex of $H$. Let $N(v)$ be the set of all vertices of $H$ adjacent to $v$.
%----------------------------------------------------------------------------------------------------------------------------------
\begin{theorem} \cite[Proposition 6.2.4]{V}\label{V2}
If $H$ is Cohen-Macaulay and $v$ is a vertex of $H$, then $H \setminus (v, N(v))$ is Cohen-Macaulay.
\end{theorem}
%--------------------------------------------------------------------------------------------------------------------------------------
\begin{theorem} \cite{Z} \label{Z1}
Let $G$ be an $r$-partite unmixed graph such that all maximal cliques are of size $r$. Then all parts have the same cardinality and there
is a perfect $2$-matching between each two parts.
\end{theorem}
%----------------------------------------------------------------------------------------------------------------------------------
Now, we present the main theorem of this paper which is generalization of \cite[Theorem 2.4]{EV}.
%----------------------------------------------------------------------------------------------------------------------------------------
\begin{theorem}\label{C}
Let $G$ be an $r$-partite graph in the class $\mathcal{G}$ such that each maximal clique is of size $r$. If $G$ is Cohen-Macaulay then there is a vertex of degree $r-1$ in $G$.
\end{theorem}
\begin{proof}
By Theorem \ref{Z1} all parts have the same cardinality. So there is a positive integer $n$ such that $|V|=rn$. Assume that for all vertices
$v$ in $G$ we have $\deg(v)\geq r$. Let  $Q_i=\{x_{1i},x_{2i},\ldots ,x_{ri} \}$ for $i=1,\ldots,n$ are cliques in a basic clique cover of $G$. Without loss of generality, assume that $v_{11}$ be a vertex of the minimal degree. If $\deg(v_{11})=(r-1)n$ then $G=K_{n,n,\ldots,n}$ is a complete
$r-$partite graph. Thus $G$ is not Cohen-Macaulay by \cite[Exercise 5.1.26]{HB} and we get a contradiction. Therefore,  $r \leq \deg(v_{11}) \leq (r-1)n-1$.

Let $N(v_{11})=\{v_{21},\ldots,v_{2l_2},v_{31},\ldots,v_{3l_3},\ldots,v_{r1},\ldots,v_{rl_r}\}$. We have $\deg(v_{11})=l_2+\cdots+l_r$. Without loss of generality, we may assume that $l_2 \leq l_i$ for $i=3,\ldots,r$. Set $G'=G \setminus (\{v_{11}\},N(v_{11}))$.
The graph $G'$ is Cohen-Macaulay by  Theorem \ref{V1}. If $l_2=1$, then, there exists $3 \leq i \leq r$ such that $l_i \geq 2$. The sets
\begin{eqnarray*}
\{v_{12}, \ldots, v_{1n}, v_{22}, \ldots, v_{2n}, v_{3(l_3+1)}, \ldots, v_{3n}, \ldots,
\widehat{\newblock{(v_{i(l_i+1)},\ldots,v_{in})}},
\ldots,&& \\ v_{r(l_r+1)}, \ldots, v_{rn}\} &&
\end{eqnarray*}
and
$$\{v_{12}, \ldots, v_{1n}, v_{3(l_3+1)}, \ldots, v_{3n}, \ldots, v_{i(l_i+1)}, \ldots, v_{in}, \ldots, v_{r(l_r+1)}, \ldots, v_{rn}\}$$
are two minimal vertex covers for $G'$ and their cardinalities are not equal. Here, by $\widehat{(v_{i(l_i+1)},\ldots,v_{in})}$ we mean the vertices $v_{i(l_i+1)}, \ldots, v_{in}$ are removed from the set. This contradicts to
Cohen-Macaulayness of $G'$. Therefore, $l_2 \geq 2$.  We claim that
$$\deg(v_{1i}) = l_2 + l_3 + \cdots + l_r = \deg(v_{11}), \ \ \ \ i=1,\ldots,l_2.$$
It is enough to show that  $\deg(v_{12})=l_2+l_3+\cdots+l_r$ and analogous argument proves the claim. If $\deg(v_{12}) > l_2+l_3+\cdots+l_r$, then there is a $j_t$, $l_t+1 \leq j_t \leq n$ for some $2 \leq t  \leq r$, such that $v_{12} \sim v_{tj_t}$. Without loss of generality we assume that $t=2$.

If there is $j_2$, $l_2+1 \leq j_2 \leq n$, such that $v_{12} \sim v_{2j_2}$ then there is a minimal vertex cover for $G'$ containing the set
$$\{v_{12}, v_{1(l_2+1)}, \ldots, v_{1n}, v_{3(l_3+1)}, \ldots, v_{3n}, \ldots, v_{r(l_r+1)}, \ldots, v_{rn}\}.$$
In other hand,  $\{v_{2(l_2+1)}, \ldots, v_{2n}, \ldots, v_{r(l_r+1)}, \ldots, v_{rn}\}$ is a minimal vertex cover of $G'$.
By $l_2 \geq 2$ and Theorem~\ref{V1}, this contradicts Cohen-Macaulayness of $G'$.  Therefore $\deg(v_{12})= l_2 + l_3 + \cdots + l_r$. Thus, for all $1 \leq i \leq l_2$ we have
$N(v_{1i})=\{v_{21}, \ldots, v_{2l_2}, v_{31}, \ldots, v_{3l_3}, \ldots, v_{r1}, \ldots, v_{rl_r}\}$. Consider the graph
$H = G\setminus \big(\{v_{2(l_2+1)}, \ldots, v_{2n}, \ldots, v_{r(l_r+1)}, \ldots, v_{rn}\} \cup N(v_{2(l_2+1)}) \cup \cdots \cup N(v_{2n}) \cup \cdots \cup  N(v_{r(l_r+1)}) \cup \cdots \cup N(v_{rn}) \big)$.  By Theorem \ref{V2}, $H$ is Cohen-Macaulay but the complement of $H$ is not connected. This is a contradiction by  \cite[Exercise 5.1.26]{HB}.
\end{proof}
%-----------------------------------------------------------------------------------------------------------------------------------
Theorem \ref{C} implies that the perfect $r$-matching in a Cohen-Macaulay $r$-partite graph is unique.
%-----------------------------------------------------------------------------------------------------------------------------------
\begin{corollary}
Let $G$ be an $r$-partite graph in the class $\mathcal{G}$ such that all maximal cliques are of size $r$. If $G$ is Cohen-Macaulay then there is a unique perfect $r$-matching in $G$.
\end{corollary}
\begin{proof}
Since $G$ is in the class $\mathcal{G}$, there is a perfect $r$-matching in $G$. By Theorem \ref{C}, there is a vertex $x \in V(G)$ of degree $r-1$. Therefore, the $r$-clique in the $r$-matching which contains $x$, must be in all perfect $r$-matchings of $G$. The graph $G \setminus (\{x,N(x)\})$ is again an $r$-partite graph in the class $\mathcal{G}$ which is Cohen-Macaulay by Theorem \ref{V2}. Continuing this process, we find that the chosen perfect $r$-matching is the unique perfect $r$-matching in $G$.
\end{proof}

%--------------------------------------------------------------------------------------------------------------------

\end{document}